\documentclass[11pt,a4paper]{article}

\usepackage[english]{babel}
\usepackage{pst-grad} 
\usepackage{pst-plot} 
\usepackage{pstricks}
\usepackage{amsmath,amssymb,mathrsfs,amsthm}
\usepackage{tikz} 
\usepackage{graphicx}

\newcommand{\HH}{\mathcal H}

\newcommand{\RR}{\mathbb R}

\newcommand{\TT}{\mathbb T}

\def\jump#1{{[\hspace{-2pt}[#1]\hspace{-2pt}]}}
\def\norm#1{\left\vert \left\vert #1  \right\vert \right\vert }
\def\seminorm#1{ \left\vert #1  \right\vert }

\newcommand{\St}{\mathcal{S}}

\def\tnorm#1{\left\vert \left\vert \left\vert #1  \right\vert \right\vert \right\vert}

\usepackage[bookmarks,colorlinks]{hyperref}
\hypersetup{colorlinks=true, linkcolor = blue,
	anchorcolor = blue,
	citecolor = blue,
	filecolor = blue,
	urlcolor = blue}

\newtheorem{lem}{Lemma}
\newtheorem{corollary}{Corollary}

\newtheorem*{prop-non}{Proposition}
\newtheorem{theo}{Theorem}
\newtheorem*{theorem-non}{Theorem}

\usepackage[a4paper]{geometry}
\title{On the global well-posedness of interface dynamics\\
for gravity Stokes flow}
\author{
	Francisco Gancedo
	\\{\footnotesize Departamento de An\'alisis Matemático \& IMUS}
	\\{\footnotesize Universidad de Sevilla}
	\\{\footnotesize Sevilla, Espa\~na}
	\\{\footnotesize email: {\it fgancedo@us.es}}
	\and 
	Rafael Granero-Belinch\'on
	\\{\footnotesize Departamento de Matem\'aticas, Estad\'istica y Computaci\'on}
	\\{\footnotesize Universidad de Cantabria}
	\\{\footnotesize Santander, Espa\~na}
	\\{\footnotesize email: {\it rafael.granero@unican.es}}
	\and 
	Elena Salguero
	\\{\footnotesize Max Planck Institute for Mathematics in the Sciences}
	\\{\footnotesize Leipzig, Deutschland}
	\\{\footnotesize email: {\it elena.salguero@mis.mpg.de}}
}

\begin{document}
\maketitle 

\begin{abstract}
In this paper we establish the global-in-time well-posedness for an arbitrary $C^{1+\gamma}$, $0<\gamma<1$, initial internal wave for the free boundary gravity Stokes system in two dimensions. This classical well-posedness result is complemented with a weak solvability result in the case of $C^\gamma$ or Lipschitz interfaces. Furthermore, we also propose and study several one-dimensional models that capture different features of the full internal wave problem for the gravity Stokes system and show that all of them present finite time singularities. This fact evidences the fine structure of the non-linearity in the full system that allows for the free boundary problem to be globally well-posed while simplifications of it blow-up in finite time. 
\end{abstract}

{\small
\tableofcontents}
\section{Introduction and main results}\label{s:introduction} 
The study of the dynamics of free boundary problems dates back, at least, to the works of Laplace (1776) and Lagrange (1781 and 1786) on the dynamics of water waves. In this problem, there is one fluid filling the time-dependent domain $\Omega(t)$ bounded by the moving interface $\Gamma(t)$. This fluid is assumed to be surrounded by vacuum. 

Since these pioneer works, many celebrated researchers have studied the case of an inviscid fluid with a free boundary. In particular, today is well-known that the one-phase Euler equations with moving free-surface are locally well-posed in Sobolev spaces if the pressure function satisfies the so-called Rayleigh-Taylor (RT) stability condition, that requires the sign of the normal derivative of the pressure to be negative, namely, 
$$
\left.\frac{\partial p}{\partial n(t)}\right|_{t=0} < 0
$$
on the initial surface $\Gamma(0)$ (see \cite{cordoba_interface_2010} and the references therein).

When two inviscid fluids are considered, the situation changes drastically as the dynamics naturally develop instabilities (\cite{lannes2013} and the references therein). In fact, the two-phase Euler equations are used as the basic model for the Rayleigh-Taylor and Kelvin-Helmholtz instabilities for fluids with different densities under gravity forces and fluids with  the same density but discontinuous velocities at the interface $\Gamma(t)$, respectively. Perturbations of the flat interface in the linear regime grow; the  nonlinearity is unable of stabilizing such a growth. As a result, this nonlinear two-fluid problem where the two fluids have constant (different) densities $\rho^+$ and $\rho^-$
\begin{subequations}\label{Euler}
\begin{alignat}{2}
\rho^\pm(u_t^\pm+(u^\pm\cdot\nabla)u^\pm)&=-\nabla p^\pm-(0,\rho^\pm)^T&&\quad\text{in}\quad \Omega^\pm(t)\,\,,\\
\nabla\cdot u^\pm&=0 &&\quad\text{in}\quad \Omega^\pm(t)\,\,,\\
\jump{p}&=0&&\quad \text{on}\quad \Gamma(t),\\
z_t&=u(z,t)&&\quad \text{on}\quad \Gamma(t),
\end{alignat}
\end{subequations}
where $\jump{f}=f^+-f^-\text{ on}\quad \Gamma(t)$, is ill-posed in Sobolev spaces when surface tension effects are neglected \cite{lannes2013}. 

A similar problem arises when the fluids fill a porous medium. In this case, the problem is known in the literature as the Muskat/interface Hele-Shaw problem \cite{Alazard,castro2013breakdown,castro2016,castro2012rayleigh,cordoba_contour_2007}. There, the system under study is given by Darcy's law
\begin{subequations}\label{Darcy}
\begin{alignat}{2}
u^\pm&=-\nabla p^\pm-(0,\rho^\pm)^T,&&\quad\text{in}\quad \Omega^\pm(t)\,\,,\\
\nabla\cdot u^\pm&=0 &&\quad\text{in}\quad \Omega^\pm(t)\,\,,\\
\jump{p}&=0&&\quad \text{on}\quad \Gamma(t),\\
z_t&=u(z,t)&&\quad \text{on}\quad \Gamma(t).
\end{alignat}
\end{subequations}
The two-phase Muskat problem can also be ill-posed if the RT condition is not satisfied \cite{cordoba_contour_2007} while the one-phase Muskat problem (in an unbounded domain, at least) satisfies the RT condition automatically. Furthermore, besides its own mathematical interest, the Muskat problem has been a successful benchmark problem for the free boundary Euler equations and conversely. In particular, the ideas leading to turning waves for the two-phase Muskat were also implemented in the case of water waves \cite{castro2012rayleigh} and the splash singularities in water waves \cite{castro2013finite} were later extended to the one-phase Muskat problem \cite{castro2016}.

Even if for many applications, the inviscid approximation is enough in practice, a better, more careful study requires the inclusion of viscosity effects (see \cite{granero-belinchon_well-posedness_2021} and the references therein). In this case, one has to study the free boundary problem for the Navier-Stokes equations
\begin{subequations}\label{NavierStokes}
\begin{alignat}{2}
\rho^\pm(u_t^\pm+(u^\pm\cdot\nabla)u^\pm)-\Delta u^\pm&=-\nabla p^\pm-(0,\rho^\pm)^T&&\quad\text{in}\quad \Omega^\pm(t)\,\,,\\
\nabla\cdot u^\pm&=0 &&\quad\text{in}\quad \Omega^\pm(t)\,\,,\\
\jump{(\nabla u +\nabla u^T-p\, \text{Id})n}&=0&&\quad \text{on}\quad \Gamma(t),\\
\jump{u}&=0&&\quad \text{on}\quad \Gamma(t),\\
z_t&=u(z,t)&&\quad \text{on}\quad \Gamma(t).
\end{alignat}
\end{subequations}
One of the classical references on such a system is \cite{Denisova}.  However, in the last year there has been some works studying the free boundary problem for such a system both in the two-phase and one-phase case \cite{NSdensity,guo_almost_2013,guo2018stability}.

After going to dimensionless variables and the Boussinesq approximation, system \eqref{NavierStokes} reads
\begin{subequations}\label{NavierStokes2}
\begin{alignat}{2}
\frac{1}{Pr}(u_t^\pm+(u^\pm\cdot\nabla)u^\pm)-\Delta u^\pm&=-\nabla p^\pm-(0,Ra\,\rho^\pm)^T&&\quad\text{in}\quad \Omega^\pm(t)\,\,,\\
\nabla\cdot u^\pm&=0 &&\quad\text{in}\quad \Omega^\pm(t)\,\,,\\
\jump{(\nabla u +\nabla u^T-p\, \text{Id})n}&=0&&\quad \text{on}\quad \Gamma(t),\\
\jump{u}&=0&&\quad \text{on}\quad \Gamma(t),\\
z_t&=u(z,t)&&\quad \text{on}\quad \Gamma(t),
\end{alignat}
\end{subequations}
where $Pr$ and $Ra$ are the Prandtl and Rayleigh dimensionless numbers \cite{grayer_ii_dynamics_2022,leblondthesis}. See \cite{DZ,BoussinesqDensity} for recent works on this problem. As a consequence of the model, in the asymptotic regime $1\ll Pr$, a good approximation is the two-phase Stokes system
\begin{subequations}\label{Stokes}
\begin{alignat}{2}
-\Delta u^\pm&=-\nabla p^\pm-(0,\rho^\pm)^T&&\quad\text{in}\quad \Omega^\pm(t)\,\,,\\
\nabla\cdot u^\pm&=0 &&\quad\text{in}\quad \Omega^\pm(t)\,\,,\\
\jump{(\nabla u +\nabla u^T-p\text{Id})n}&=0&&\quad \text{on}\quad \Gamma(t),\\
\jump{u}&=0&&\quad \text{on}\quad \Gamma(t),\\
z_t&=u(z,t)&&\quad \text{on}\quad \Gamma(t),
\end{alignat}
\end{subequations}
Furthermore, the previous system has been also derived from a microscopic formulation of sedimenting particles in a fluid \cite{hofer_2018_stokes} and from a Vlasov-Stokes kinetic system \cite{hofer_2018_stokes2}. This system and close variants (as the surface tension case) have been studied by many different researchers in the previous years \cite{free_boundary,cobb,grayer_ii_dynamics_2022,leblond,matioc_two-phase_2021,matioc_two-phase_2022,matioc_capillarity_2022,mecherbet_2019_stokes,mecherbet_2021_droplet,
mecherbet_few_2022,zheng2017local}. 

In this paper we study the case of two fluid with different constant densities evolving by \eqref{Stokes} in absence of surface tension in two spatial dimensions. We observe that in our paper \cite{GGS23} we proved that system \eqref{Stokes} can be equivalently written using the following contour dynamics formulation
\begin{align} \label{eq:A}
	z_t(\alpha,t) & =  (\rho^- - \rho^+)  \int_{\TT} \St(z(\alpha,t)-z(\beta,t)) \cdot \partial_\beta z^\perp (\beta,t) z_2(\beta,t)  d\beta,
\end{align}
where the so-called $x_1$-periodic Stokeslet reads 
\begin{equation*}
	\St(y) =   \frac{1}{8 \pi } \log \left(2(\cosh (y_2)-\cos( y_1)) \right) \cdot I
	- \frac{y_2}{8 \pi (\cosh (y_2) - \cos (y_1))} \left( \begin{array}{cc}
		-\sinh (y_2) &\sin (y_1) \\
		\sin (y_1) & 	\sinh (y_2)
	\end{array} \right).
\end{equation*}
This equation resembles the vortex patch equation \cite{bertozzi1993global,deemzabuskyvortex,verdera2023vortex}
\begin{align} \label{eq:B}
	z_t(\alpha,t) & =  -\omega_0\int_{\TT} \mathcal{V}(z(\alpha,t)-z(\beta,t)) \cdot \partial_\beta z(\beta,t)d\beta,
\end{align}
where the kernel reads 
\begin{equation*}
	\mathcal{V}(y) =   \frac{1}{4 \pi } \log \left(y_1^2+y^2_2\right) \cdot I
.
\end{equation*}
Even if (part of) the kernels involved are similar, at the linear level both equations behave very differently. Indeed, the presence of the $z_2$ term in \eqref{eq:A} coming from the gravity, imposes a crucial anisotropy in the problem \eqref{Stokes}. As a consequence of such a term, the linear operator in \eqref{eq:A} has order $-1$ while the linear operator in \eqref{eq:B} has order $0$.

A number of well-posedness results for the two-dimensional case without surface tension are available in the literature. On the one hand, for \eqref{Stokes} in a planar (and also three-dimensional) domain bounded in the vertical direction, Leblond established the global existence and uniqueness for bounded initial density \cite{leblond,leblondthesis}. In particular, such results cover the case of density patch described above. The same case but in a bounded planar domain was also proved by Antontsev, Yurinsky and Meirmanov \cite{free_boundary} for the case of a $C^2$ initial interface. However, the inherent anisotropy of the system, where the vertical direction plays a crucial role compared to the horizontal one, makes the case with an unbounded domain out of reach of the previous references. In the case of the whole plane, the only available well-posedness result is the one by Grayer II that is able to prove the global existence of solution for initial density that is in $L^\infty$ and with one domain bounded \cite{grayer_ii_dynamics_2022}. The hypothesis on the domain is crucial as it serves to control it due to the transport character of the conservation of mass equation for the density. In fact, such a bounded domain is propagated by the transport structure of the problem and allows the author to handle the logarithmic growth of the Stokeslet kernel. Part of our analysis relies on a new cancellation found in the Stokeslet that ensures that the velocity $u$ in our case has finite $L^p$ energy instead of linear or logarithmic growth.

Our main result in this paper is the following:
\begin{theo}[Global well-posedness of two-phase Stokes]\label{teo1}
	Let $z_0(\alpha) \in C^{1+\gamma}(\TT)$ be the initial data satisfying the arc-chord condition for the two-phase Stokes problem \eqref{Stokes}. Then, there exists a solution 
	$$ z(\alpha,t) \in C([0,T];C^{1+\gamma})$$
		for any $T>0$.
\end{theo}
We observe that the arc-chord condition cannot blow up and that the initial data has arbitrary size. The proof of this result is based on new bounds for the velocity field for the case of density patches that exploits the structure of the kernels involved \cite{GGS23}. Once this regularity is controlled, using the contour dynamics equation for a general curve $z(\alpha,t)$ given in \cite{GGS23}, any higher order regularity can be controlled ($C^{k+\gamma}$ with $k\geq 2$).

The main importance of our contribution is three-fold. First, our result is the first one handling the case where the fluids fill an unbounded-in-$x_2$ domain without imposing at the same time some integrability or compactly support condition on the densities. This is a main difficulty in any rigorous proof of the result due to the growth of the Stokeslet (as $|x_2|$ in this direction) and the anisotropy of the system where the $x_2$ variable play a crucial role. Second, our result also supersedes \cite{free_boundary} in the sense that it requires barely $C^{1}$ initial data (any $C^{1+\gamma}$ is in fact enough) instead of $C^2$. In particular, our result covers initial interfaces with unbounded curvature and large slopes. Finally, our result applies regardless of the stratification of the fluids, \emph{i.e.} it proves that there is not Rayleigh-Taylor sign condition required to define global solutions and also regardless of the size and geometry of the initial data. Furthermore, as observed in \cite{leblondthesis}, the $C^{1+\gamma}$ regularity of the velocity field can be used to propagate the $C^\gamma$ or $W^{1,\infty}$ norm of the interface and, as a consequence, to define \emph{Lagrangian solutions} of low regularity. Due to the new estimates that we proved for the velocity such Lagrangian solutions are also weak solutions of the system. We refer to \cite{inversi} for a study between the different concepts of solutions.

Now we turn our attention to several one-dimensional models of an internal wave in gravity Stokes flow (see Section 3). In particular, we propose the following nonlocal and nonlinear partial differential equations:
\begin{subequations}\label{quadratic}
	\begin{align} 
		u_t + u u_\alpha =- \Lambda^{-1} (u) & \quad \alpha\in \TT, \, t\in[0,T],\\
		u_0(\alpha,0) = u_0(\alpha) & \quad   \alpha \in \TT ,
	\end{align}
\end{subequations}

\begin{subequations} \label{cubiclocal}
	\begin{align}
		u_t + u^2 u_\alpha =- \Lambda^{-1} (u)&\quad \alpha\in \TT, \, t\in[0,T],\\
		u_0(\alpha,0) = u_0(\alpha)& \quad \alpha \in \TT ,
	\end{align}
\end{subequations}

\begin{subequations}\label{cubicnonlocal2}
	\begin{align}
		u_t + \frac{1}{2}  \mathcal{H}\left( u^2\right) u_\alpha = - \Lambda^{-1} (u)&\quad \alpha\in \TT, \, t\in[0,T],\\
		u_0(\alpha,0) = u_0(\alpha)& \quad \alpha \in \TT .
	\end{align}
\end{subequations}

The operator $\Lambda^{-1} (u)$ in the periodic one-dimensional torus is defined as 
\begin{equation}\label{lambda-1}
	\Lambda^{-1} (u) \, (\alpha)= - \frac{1}{2\pi} \int_{-\pi}^\pi \log \left(4 \sin^2 \left(  \frac{\alpha-\beta}{2}\right) \right) u(\beta) d \beta.
\end{equation}
This operator was studied in more detail in the authors' previous work \cite{GGS23}.
The Hilbert transform in the periodic one-dimensional torus is defined as
\begin{equation}\label{hilbert}
	\HH(u)(\alpha) = \frac{1}{2\pi} \int_{-\pi}^\pi \cot \left(\frac{\alpha-\beta}{2} \right) u(\beta) d\beta. 
\end{equation}
Also, related to the previous two operators, the fractional Laplacian $\Lambda u = (-\Delta)^{1/2}$ is defined as 
\begin{equation}\label{lambda}
	\Lambda(u)(\alpha) = \frac{1}{2\pi} \int_{-\pi}^\pi  \frac{u(\alpha)-u(\beta)}{2\sin^2 \left(  \frac{\alpha-\beta}{2}\right)}   d\beta. 
\end{equation}
Note that 
$$ \partial_\alpha\Lambda^{-1 } (u) = -\HH(u), \quad \partial_\alpha \HH(u) = \Lambda(u).$$

For all these one-dimensional systems we prove the following (roughly stated) result \footnote{See below for a more precise statement.}:

\begin{theo}[Global well-posedness vs. finite time singularities for 1D models]\label{teo2}
Let $u_0\in H^n$ for $n$ high enough be the initial data for \eqref{quadratic}, \eqref{cubiclocal} or \eqref{cubicnonlocal2}. Then, if $u_0$ is small enough in appropriate Sobolev spaces the solution emanating from it is globally defined. At the contrary, if the initial data $u_0$ is large enough, the solution blows up in finite time. 
\end{theo}

The finite-time blow-up for the 1D models is attained following the ideas in \cite{AC2021,CCF2005,CCF2006,DDL2009}.

In this sense, our second result shows that the fine structure of the nonlinearity is required to have globally defined smooth solutions for the free boundary. Indeed, we show that when only some of the nonlinear interactions are considered, the linear damping is unable to dominate the dynamics and large solutions blow up.

\section{Global well-posedness of the internal wave for the gravity Stokes system}
As in our previous work \cite{GGS23}, we consider two incompressible, viscous and immiscible fluids filling the domain $\TT \times \RR$. Both fluids are separated by the curve 
$$\Gamma(t)=\{ (z_1(\alpha,t),z_2(\alpha,t)); \quad \alpha \in [-\pi,\pi], \quad z(\alpha+ 2 \pi k,t) = (2 \pi k,0) + z(\alpha,t)  \}.$$
Then the upper fluid fills the domain $\Omega^{+}(t)$, while the lower fluid lies in $\Omega^{-}(t)$. Thus, the problem that we consider reads 
\begin{subequations}\label{stokes}
	\begin{align}
		-\Delta u^\pm&=-\nabla p^\pm-\rho^\pm(0,1)^t,&&\quad x\in\Omega^\pm(t),t\in[0,T],\\
		\nabla\cdot u^\pm&=0,&&\quad x\in\Omega^\pm(t),t\in[0,T],\\
		\jump{-pI + \mu ((\nabla u +(\nabla u)^T)/2)}\cdot(\partial_\alpha z)^\perp&= 0,&&\quad x\in\Gamma(t),t\in[0,T],\\
		\jump{u}&=0,&&\quad x\in\Gamma(t),t\in[0,T],\\
		z_t&= u(z,t), &&\quad  t\in[0,T],\\
		z&= z_0, &&\quad t =0 .
	\end{align}
\end{subequations}
We observe that there exists a constant $M$ big enough such that 
$$
\TT \times[M,+\infty)\subset\Omega^+(t)\quad\mbox{and}\quad
\TT \times(-\infty,-M]\subset\Omega^-(t).
$$

We recall that the velocity field can be expressed as a convolution of the density function with the so-called $x_1$-periodic Stokeslet
\begin{align} \label{eq:1}
	u & = {(\nabla^\perp \Delta^{-2} \nabla^\perp)} (0,\rho)^t = \St \ast (0,\rho)^t
\end{align}
where
\begin{equation}\label{eq:stokeslet}
	\St(y) =   \frac{1}{8 \pi } \log \left(2(\cosh (y_2)-\cos( y_1)) \right) \cdot I
	- \frac{y_2}{8 \pi (\cosh (y_2) - \cos (y_1))} \left( \begin{array}{cc}
		-\sinh (y_2) &\sin (y_1) \\
		\sin (y_1) & 	\sinh (y_2)
	\end{array} \right).
\end{equation}
See our previous work \cite{GGS23} for a rigorous derivation of $\St$. A direct calculation shows that the components $S_{1,1}(y)$ growth as $|y_2|$ if $y_2\to\infty$.
\begin{lem}\label{lem:integrable}
The Stokeslet applied to the gravity force is integrable in the vertical strip $\mathbb{T}\times\mathbb{R}$.
\end{lem}
\begin{proof}
	From the explicit expression of $\St$ in \eqref{eq:stokeslet}, 
	\begin{align*}
		u_1(x,t) &= - \frac{1}{8\pi}\int_{\TT \times \RR} \frac{y_2 \sin (y_1)}{ \cosh(y_2)-\cos(y_1)} \rho(x-y,t)dy, \\
		u_2(x,t) & = \frac{1}{8\pi} \int_{\TT \times \RR} \log \left(2(\cosh (y_2)-\cos( y_1)) \right) \rho(x-y,t)  dy \\
		& \quad -  \frac{1}{8\pi} \int_{\TT \times \RR} \frac{y_2 \sinh (y_2)}{ \cosh(y_2)-\cos(y_1)} \rho(x-y,t)dy, \\
		& = \frac{1}{8\pi} \int_{\TT \times \RR} \left( \log \left(2(\cosh (y_2)-\cos( y_1)) \right)  -\frac{y_2 \sinh (y_2)}{ \cosh(y_2)-\cos(y_1)} \right)  \rho(x-y,t) dy .
	\end{align*}
	In compact notation, 
	$$ u = \tilde\St \ast \rho,$$
	where
	$$ \tilde\St_1(y) = \frac{y_2 \sin (y_1)}{ \cosh(y_2)-\cos(y_1)}$$ 
	and 
	$$ \tilde\St_2(y) =  \left( \log \left(\cosh (y_2)-\cos( y_1) \right)  -\frac{y_2 \sinh (y_2)}{ \cosh(y_2)-\cos(y_1)} \right).$$
	At this point, we will study the behavior of these kernels at $0$ and $\infty$, where they can develop singularities.
	
	In a neighborhood $B_0$ of  $y = 0$, the $x_1$-periodic Stokeslet behaves like the classical Stokeslet in the 2D plane and it is easy to check that it is integrable (it behaves as $\log|y|$), so that in particular 
	$$ \norm{\tilde\St_1}_{L^1(B_0)} < \infty, \quad \norm{\tilde\St_2}_{L^1(B_0)} < \infty.$$
	
	At $|y_2| \to \infty$, it is clear that
	$$ |\tilde\St_1(y)| \lesssim e^{-|y_2|} .$$
	Besides, when $y_2 \to +\infty$,
	$$ \tilde\St_2(y) \sim -\cos(y_1) \int_{y_2}^\infty  \frac{s}{\cosh (s)} ds.$$
	Notice that, if $y_1$ is fixed such that $\cos(y_1) \neq 0$,
	$$ \lim\limits_{y_2 \to +\infty} \frac{\tilde\St_2(y)}{-\cos(y_1) \int_{y_2}^\infty \frac{s}{\cosh(s)} ds } =  \lim\limits_{y_2 \to +\infty} \frac{\partial_{y_2}\tilde\St_2(y)(y_2)}{\cos(y_1) \frac{y_2}{\cosh(y_2)}} = 1,$$
	where 
	\begin{align*}
		\partial_{y_2}\tilde\St_2(y) &= -\frac{y_2\cosh(y_2)}{\cosh(y_2)-\cos(y_1)} +\frac{y_2\sinh^2(y_2)}{(\cosh(y_2)-\cos(y_1))^2} \\
		& = -\frac{y_2(1+\cosh(y_2)\cos(y_1))}{(\cosh(y_2)-\cos(y_1))^2}.
	\end{align*}
	When $\cos(y_1) = 0$, 
	$$\lim\limits_{y_2 \to +\infty} \frac{F_2(y_2)}{ \int_{y_2}^\infty \frac{s}{\cosh(s)} ds } =  0.$$
	This translates into exponential decay in $y_2$ of $\tilde\St_2(y_2)$. Namely, 
	$$ |\tilde\St_2(y)| \lesssim  e^{-\frac{y_2}{2}} .$$
	The same analysis holds for the case $y_2 \to -\infty$.
	Therefore, we have proved that 
	\begin{equation}\label{L1kernel}
		\norm{\tilde\St_1}_{L^1(\TT \times \RR)} < \infty, \quad   \norm{\tilde\St_2}_{L^1(\TT \times \RR)} < \infty.
	\end{equation}
\end{proof}

This particular behavior of the kernel for the gravity forcing will have consequences in the boundedness of the velocity field and its derivatives, which at the same time will imply regularity of the free interface. With this in mind, we prove the following result:
\begin{lem}\label{lem:velintegrable}
The velocity $u$ solving \eqref{Stokes} is integrable in our vertical strip $\mathbb{T}\times\mathbb{R}$.
\end{lem}
\begin{proof}
We split the density as
	\begin{equation*}
		\rho = \rho^c(x_1,x_2,t) + \rho^\infty(x_2,t),
	\end{equation*}
where 
\begin{equation*}
	\rho^c = \left\{
	\begin{array}{ll}
		\rho^+, & z_2(\alpha,t) \leq x_2 \leq 2 \norm{z_2(t)}_{L^\infty}, \\
		\rho^-, &  -2 \norm{z_2(t)}_{L^\infty} \leq x_2 \leq z_2(\alpha,t), \\
		0 & \text{otherwise}
	\end{array}
	  \right.
\end{equation*}
and 
\begin{equation*}
	\rho^\infty = \left\{
	\begin{array}{ll}
		\rho^+, & x_2 > 2 \norm{z_2(t)}_{L^\infty},\\
		\rho^-, &  x_2 < -2 \norm{z_2(t)}_{L^\infty}, \\
				0 & \text{otherwise}.
	\end{array}
	\right.
\end{equation*}
Then, at the level of the stream function,
\begin{equation*}
	 \Delta^2 \psi =  \partial_{x_1} \rho^c + \partial_{x_1} \rho^\infty.
\end{equation*}
Note that $\partial_{x_1} \rho^\infty = 0$ in the weak sense, as $\rho_\infty$ only depends on $x_2$ and $t$. Then, the velocity becomes
\begin{equation}\label{eq:vel}
	u(x,t) = \tilde\St \ast  \rho^c(x,t). 
\end{equation}
This trick can also be seen adding $\rho^{\infty}$ to the pressure as $\nabla q(x,t)=(0,\rho^\infty(x_2,t))$, for $q$ a continuous function.
As a consequence,
$$
\|u\|_{L^p}\leq \|\rho_0\|_{L^\infty}\|\tilde\St\|_{L^1}(4\norm{z_2(t)}_{L^\infty})^{1/p}.
$$
\end{proof}
The shape of the gravity force allows to deal with the growth of the two-dimensional Stokeslet. 

Equipped with these estimates, let us proceed to prove the main result of this paper.
\begin{proof}[Proof of Theorem \ref{teo1}]
	
	Let us consider a Lagrangian approach: let $X(a,t)$ be the trajectory with initial data $X(a,0) = a$ driven by the velocity field $u(x,t)$ as in \eqref{eq:1}. Then, the free boundary is transformed as 
	$$ X(z_0(\alpha),t) = z(\alpha,t).$$
	The velocity of the trajectories is 
	\begin{equation*}
		 \frac{d}{dt}X(a,t) = u(X(a,t),t) 
	\end{equation*}
	and the gradient evolves as
	$$ \frac{d}{dt} \nabla_a X(a,t) = \nabla u(X(a,t),t) \nabla_a X(a,t).$$
	Consequently, 
	\begin{equation}\label{eq:X}
		\|X(a,t)\|_{L^\infty}\leq \|X_0(a)\|_{L^\infty} + \int_0^t\|u(X(a,s),s)\|_{L^\infty}ds,
	\end{equation}
\begin{equation}\label{eq:gradX}
		  \exp \left( -\int_0^t \norm{\nabla u(s)}_{L^\infty}  ds  \right)\leq \norm{\nabla X(t)}_{L^\infty} \leq  \exp \left( \int_0^t \norm{\nabla u(s)}_{L^\infty}  ds  \right) 
\end{equation}
	and similarly
	\begin{align}\label{eq:gammaX}
		\frac{d}{dt}\seminorm{\nabla X}_{\gamma} &\leq \norm{\nabla u}_{L^\infty} \seminorm{\nabla X}_{\gamma} + \seminorm{\nabla u}_{\gamma} \norm{\nabla X}_{L^\infty}^{1+\gamma} \nonumber \\
		& \leq \norm{\nabla u}_{L^\infty} \seminorm{\nabla X}_{\gamma} +  \seminorm{\nabla u}_{\gamma} \left( \ \exp \left(  \int_0^t \norm{\nabla u(s)}_{L^\infty}  ds \right)\right)^{1+\gamma}.
	\end{align}

Control of the supremum of the free boundary is given by control of $\norm{u}_{L^\infty}$, due to \eqref{eq:X}. Indeed, we only need to control the second component $X(z_0(\alpha),t)_2 = z_2(\alpha,t)$ in $L^\infty$, which is the only one that could be unbounded, namely
$$\frac{d}{dt} z_2(\alpha,t) = u_2(z(\alpha,t),t),$$
so  that
$$\norm{z_2(t)}_{L^\infty} \leq \norm{z_2(0)}_{L^\infty}+\int_0^t\norm{u(s)}_{L^\infty}ds.$$
We observe that Lemma \ref{lem:integrable} gives the uniform control 
\begin{equation}\label{eq:uinfty}
	\norm{u}_{L^\infty} \lesssim \norm{\tilde\St}_{L^1} \norm{\rho_0}_{L^\infty},
\end{equation}
thus
$$ \norm{z_2}_{L^\infty} \leq \norm{z_0}_{L^\infty}+ C \norm{\rho_0}_{L^\infty}t. $$

Now, we note from \eqref{eq:gradX} and \eqref{eq:gammaX}, that control of $\norm{\nabla u}_{L^\infty}$ and $\seminorm{\nabla u}_{\gamma}$ translates into control of $\seminorm{X}_{C^{1+\gamma}}$. Furthermore, the lower bound on the derivative of the trajectories ensures that the arc-chord condition does not collapse.

Firstly, 
\begin{align*}
	\norm{\nabla u}_{L^\infty}  \leq \norm{\nabla u}_{W^{1,p}} = \norm{\nabla u}_{L^p} + \norm{\nabla^2 u}_{L^p}
\end{align*}
and 
$$  \seminorm{\nabla u}_{\gamma} \leq  \norm{\nabla^2 u}_{L^p}$$
for $2 < p < \infty$ and $ \gamma = 1-\frac{2}{p}$.
Furthermore, 
$$ \norm{\nabla^2 u}_{L^p} = \norm{\nabla^2 \partial_{x_1} \nabla^{\perp} \Delta^{-2} (\rho^c)}_{L^p} \leq C\norm{\rho^c}_{L^p} \leq C(\norm{\rho_0}_{L^\infty}) \norm{z_2}_{L^\infty}^{1/p}\leq C(1+t)^{1/p},$$
by Calderon-Zygmund theory, since $\nabla^2  \nabla^{\perp} \Delta^{-2} \nabla^\perp$ is a 0-th order kernel. The $L^p$ norm of the gradient of the velocity is estimated by Gagliardo-Niremberg interpolation as 
$$ \norm{\nabla u}_{L^p} \lesssim \norm{u}_{L^2}^{4p/(5p-2)} \norm{\nabla^2 u }_{L^p}^{p/(5p-2)}\leq C(1+t)^{1/p}$$
so the control of $	\norm{\nabla u}_{L^\infty}$ and  $\seminorm{\nabla u}_{\gamma}$ holds for every positive time $t >0$. 

We can conclude then that 
$$
X(a,t) \in C([0,T]; C^{1+\gamma})
$$
and the associated free boundary
$$ z(\alpha,t) \in C([0,T]; C^{1+\gamma})$$
for  $0 <\gamma <1$ and any $T>0$, which provides the existence of global regular solutions and in particular of the nonlocal contour dynamics system given in \eqref{eq:A}.
\end{proof}

\begin{corollary}Let $z_0(\alpha)$ be the initial data satisfying the arc-chord condition for the two-phase Stokes problem \eqref{Stokes}. Assume that $z_0(\alpha) \in C^{\gamma}(\TT)$, then, there exists a weak solution 
	$$ z(\alpha,t) \in C([0,T];C^{\gamma})$$ for any $T>0$. Similarly, assume that $z_0(\alpha) \in W^{1,\infty}(\TT)$, then, there exists a weak solution 
	$$ z(\alpha,t) \in C([0,T];W^{1,\infty})$$ for any $T>0$.
\end{corollary}
\begin{proof}
As noted in \cite{inversi}, any Lagrangian density patch solution with integrable velocity is also a weak (distributional) solution. Thus, to obtain the previous result it is enough mollify the initial data $z_0$ and invoke Theorem \ref{teo1} and Lemma \ref{lem:velintegrable}.
\end{proof}

\section[Global well-posedness vs. finite time singularities in several one-dimensional models]{Global well-posedness vs. finite time singularities in several one-dimensional models of the internal waves for the gravity Stokes system}

System \eqref{quadratic} appears as an asymptotic model of Euler alignment for a specific kernel. Indeed, the Euler alignment system \cite{ST} reads
\begin{align*}
	\rho_t+(v\rho)_\alpha&=0,\\
	v_t+vv_\alpha&=\int \phi(x-y)(v(y)-v(\alpha))\rho(y)dy.
\end{align*}
If we now make the ansatz
$$
\rho=1+\varepsilon^2h, \quad v=\varepsilon u,
$$
we find
\begin{align*}
	\varepsilon^2 h_t+\varepsilon(u(1+\varepsilon^2 h))_\alpha&=0,\\
	\varepsilon u_t+\varepsilon^2 u u_\alpha&=\varepsilon\int \phi(x-y)(u(y)-u(\alpha))dy+\varepsilon^3\int \phi(x-y)(u(y)-u(\alpha))h(y)dy.
\end{align*}
Neglecting terms of order $\varepsilon^2$ we find that the previous system decouples 
\begin{align*}
	\varepsilon h_t+u_\alpha&=0,\\
	u_t+\varepsilon u u_\alpha&=\int \phi(x-y)(u(y)-u(\alpha))dy.
\end{align*}
If we now take $\phi$ the kernel associated to $\Lambda^{-1}$ we derive \eqref{quadratic}.

System \eqref{cubiclocal}  is closer to the gravity Stokes system due to the degree of the nonlinearity. Furthermore, system \eqref{cubicnonlocal2} arises from the truncation of the contour dynamics form of the gravity Stokes system up to some cubic remainder. Recall the gravity Stokes equation for a graph initial interface in the stable regime of the densities (see \cite{GGS23} for more details): 

\begin{align*}\label{eq:grafo}
	h_t(\alpha,t) &= \frac{\bar\rho}{2 \pi} \int_\TT \log \left(4 \sin^2\left(\frac{ \alpha-\beta}{2} \right) \right) h(\beta)   [ 1+h_\alpha(\alpha) h_\alpha (\beta) ]   d \beta \nonumber\\
	& \quad  +  \frac{\bar\rho}{2 \pi} \int_\TT \log \left(\frac{\sinh^2 \left(\frac{h(\alpha,t)-h(\beta,t)}{2} \right)}{\sin^2\left(\frac{ \alpha-\beta}{2} \right)}+1 \right) h(\beta)   [ 1+h_\alpha(\alpha) h_\alpha (\beta) ]   d \beta \nonumber\\
	&\quad+  \frac{\bar\rho}{2 \pi} \int_\TT  \frac{h(\beta) (h(\alpha)-h(\beta) )}{ 2 \left( \sinh^2 \left(\frac{h(\alpha,t)-h(\beta,t)}{2} \right)+\sin^2\left(\frac{ \alpha-\beta}{2} \right) \right)}   \left[ (h_\alpha(\alpha)h_\alpha(\beta)-1) \sinh(h(\alpha)-h(\beta))\right] d\beta\nonumber\\
	&\quad+  \frac{\bar\rho}{2 \pi} \int_\TT \frac{h(\beta) (h(\alpha)-h(\beta) )}{2 \left( \sinh^2 \left(\frac{h(\alpha,t)-h(\beta,t)}{2} \right)+\sin^2\left(\frac{ \alpha-\beta}{2} \right) \right)}  \left[(h_\alpha (\alpha) + h_\alpha (\beta) ) \sin (\alpha-\beta)   \right] d\beta.
\end{align*}
Here, $\bar\rho = \frac{\rho^--\rho^+}{4}$, and in particular, in the stable case $\bar\rho >0$. For simplicity, in the following we will assume $\bar\rho= 1$.

The linear contribution for this equation is only coming from the first term, namely, 
$$ \mathcal L(h) = -  \Lambda^{-1} (h)$$
The following contributions are of cubic order and they come from all the four terms. By performing a formal Taylor expansion of the nonlocal equation, we find the following cubic terms:
$$ \mathcal{C}_1(h) = \frac{1}{2} h_\alpha(\alpha) \mathcal{H} (h^2).$$

$$  \mathcal{C}_2(h) = \frac{1}{4} \frac{1}{2\pi}\int_\TT \frac{ \left(h(\alpha)-h(\beta) \right)^2}{\sin^2\left(\frac{ \alpha-\beta}{2} \right)} h(\beta) d \beta .$$

$$ \mathcal{C}_3(h)= - \frac{1}{2} \frac{1}{2\pi} \int_\TT \frac{(h(\alpha)-h(\beta))^2}{\sin^2\left(\frac{ \alpha-\beta}{2} \right)}h(\beta) d\beta  .$$
We note that
\begin{align*}
	  \mathcal{C}_2(h) +  \mathcal{C}_3(h) &= \frac{1}{4} \frac{1}{2\pi}\int_\TT \frac{(h(\alpha)-h(\beta))^3}{\sin^2\left(\frac{ \alpha-\beta}{2} \right)} d\beta   -h(\alpha)  \frac{1}{4} \frac{1}{2\pi} \int_\TT \frac{2h(\alpha)(h(\alpha)-h(\beta))-(h(\alpha)^2-h(\beta)^2)}{\sin^2\left(\frac{ \alpha-\beta}{2} \right)} d\beta \\
	 & = \frac{1}{4} \frac{1}{2\pi} \int_\TT \frac{(h(\alpha)-h(\beta))^3}{\sin^2\left(\frac{ \alpha-\beta}{2} \right)} d\beta  + h^2 \Lambda(h) - \frac{1}{2} h \Lambda(h^2).
\end{align*}
We also find the cubic term
\begin{align*}
	 \mathcal{C}_4(h)& = \frac{1}{2} \frac{1}{2\pi} \int_\TT \frac{h(\alpha)-h(\beta)}{\sin^2\left(\frac{ \alpha-\beta}{2} \right)}h(\beta) (h_\alpha (\alpha) + h_\alpha (\beta) ) \sin (\alpha-\beta)  d\beta \\
	& =  \frac{1}{2\pi} \int_\TT  \cot\left(\frac{ \alpha-\beta}{2} \right) h(\beta) (h(\alpha)-h(\beta)) (h_\alpha (\alpha) + h_\alpha (\beta) )   d\beta \\
	& =   \frac{1}{2} (h^2)_\alpha \HH(h) -  h_\alpha \HH(h^2) + \frac{1}{2} h \Lambda (h^2) - \frac{1}{3} \Lambda (h^3).
\end{align*}
All together, we have that 
$$  \mathcal{C}(h) = -\frac{1}{2} h_\alpha \HH(h^2)  + h^2 \Lambda(h)+ \frac{1}{2} (h^2)_\alpha \HH(h)  - \frac{1}{3} \Lambda (h^3) +  \frac{1}{4} \frac{1}{2\pi} \int_\TT \frac{(h(\alpha)-h(\beta))^3}{\sin^2\left(\frac{ \alpha-\beta}{2} \right)} d\beta.$$
We will select the linear term and the fist cubic term in $ \mathcal{C}(h)$ as our model. Motivation for this type of model can be seen at \cite{AC2021,CCF2005,CCF2006}.

For these three one-dimensional models, we prove global in time well-posedness for small initial data and the formation of finite time singularities under certain conditions on the initial data. 

\subsection{Global well-posedness }

\begin{theo}[Existence of solutions for small data]
	Let $u_0 \in \dot H^{4}(\TT)$ be a zero mean initial data for the Cauchy problem \eqref{quadratic}. Then, if $\norm{u_0}_{  \dot H^{4}(\TT)}$ is small enough, there exist a solution 
	$$
	u\in C([0,T],H^4), \forall \, T >0.
	$$
	Furthermore, the solution verifies
	$$
	\norm{u(t)}_{L^2}\leq C(1+t)^{-2.25}.
	$$
\end{theo}

\begin{proof} 
We will omit the time dependence of $u$ when it is clear from the context. The proof is based on finding a polynomial estimate for the quantity
	\begin{equation}
		\tnorm{u} = \sup_{t \in [0,T]} \left((1+t)^{2.25} \norm{u}_{L^2}+ \norm{u}_{{ \dot H}^4}\right). 	
	\end{equation}
		Using the Duhamel formula, we can write the solution as
		
		\begin{equation}
			u(x,t) = e^{- \Lambda^{-1} t} u_0(x) + \int_{0}^{t} e^{- \Lambda^{-1} (t-s)} u(x,s) u_x(x,s) ds.
		\end{equation}
		Taking the $L^2$ norm, we get
		\begin{align*}
			\norm{u(t)}_{L^2} & \leq  \norm{e^{- \Lambda^{-1} t} u_0}_{L^2} + \int_{0}^{t} \norm{e^{- \Lambda^{-1} (t-s)} u(s) u_x(s)}_{L^2} ds ,
			\\
			& \leq C \norm{u_0}_{ \dot H^{ 2.5}}   (1+t)^{-2.25} + C \int_{0}^{t} (1+t-s)^{-2.25} \norm{u(s) u_x(s)}_{ \dot H^
				{2.5}} ds ,
			\\
			&  \lesssim \norm{u_0}_{ \dot H^{  2.5}} (1+t)^{-2.25} + \int_{0}^{t} (1+t-s)^{-2.25} \norm{ u^2(s)}_{ \dot H^{  3.5}} ds ,
			\\
			&  \lesssim \norm{u_0}_{ \dot H^{  2.5}} (1+t)^{-2.25} + \int_{0}^{t} (1+t-s)^{-2.25} \norm{u(s)}_{L^\infty} \norm{ u(s)}_{ \dot H^{  3.5}} ds. 
\end{align*}
Using interpolation, we find that
		\begin{align*}
			\norm{u(t)}_{L^2} &  \lesssim \norm{u_0}_{ \dot H^{  2.5}} (1+t)^{-2.25} \\
			& \quad + \int_{0}^{t} (1+t-s)^{-2.25} \norm{u(s)}_{L^2}^{7/8} \norm{u(s)}_{  \dot H^4}^{1/8}\norm{u(s)}_{L^2}^{1/8} \norm{u(s)}_{ \dot  H^4}^{7/8} ds ,
			\\
			&  \lesssim \norm{u_0}_{ \dot H^{  2.5}} (1+t)^{-2.25} \\
			& \quad  + \int_{0}^{t} (1+t-s)^{-2.25} (1+s)^{-2.25}  (1+s)^{2.25}  \norm{u(s)}_{L^2} \norm{u(s)}_{  \dot H^4} ds ,
			\\
			&  \lesssim \norm{u_0}_{ \dot H^{ 2.5}} (1+t)^{-2.25} + \tnorm{u}^2 \int_{0}^{t} (1+t-s)^{-2.25} (1+s)^{-2.25}  ds ,
			\\
			&  \lesssim \norm{u_0}_{ \dot H^{ 2.5}} (1+t)^{-2.25} + \tnorm{u}^2 (1+t)^{-2.25} ,
			\\
		\end{align*}
		where we have used the decay of the linear semigroup established in \cite{GGS23}, Moser inequality , Gagliardo-Nirenberg interpolation inequality and Lemma 2.4 in \cite{elgindi_asymptotic_2017}.
		Then, 
		\begin{equation} \label{l2bound}
			(1+t)^{2.25}	\norm{u(t)}_{L^2}  \lesssim  \norm{u_0}_{ \dot H^{2.5}}+ \tnorm{u}^2  .
		\end{equation}
		
We have to bound $\norm{u}_{ \dot H^4}$ using energy estimates. 
		
		\begin{align*}
			\frac{1}{2} \frac{d}{dt} \int |\partial_x^4 u|^2 dx & = \int  \partial_x^4 u \partial_x^4 \left( -  u  u_x - \Lambda^{-1} u \right) dx ,
			\\
			& = - \int  |\Lambda^{1/2} \partial_x^3 u|^2 -  \int  \partial_x^4 u \partial_x^4 (u u_x) dx ,
			\\
			& \leq -5.5 \int (\partial_x^4 u)^2 u_x dx - 10 \int \partial_x^4 u \partial_x^3 u \partial_x^2 u dx,
			\\
			& =:  -5.5 I_1 -10  I_2.
		\end{align*}
We can bound the first contribution as follows		
		\begin{align*}
			|I_1| & \leq \norm{u_x}_{L^\infty} \norm{u}_{ \dot H^4}^2, \\
			& \lesssim  \norm{u}_{ \dot H^4}^{3/8} \norm{u}_{L^2}^{5/8}  \norm{u}_{ \dot H^4}^2 ,
			\\
			& \lesssim (1+t)^{-2.25 \frac{5}{8}} \tnorm{u}^3.
		\end{align*}
The second term can be handled similarly		
		\begin{align*}
			I_2 & = \int \partial_x^4 u \partial_x^3 u \partial_x^2 u dx ,
			\\
			& = - \int  \partial_x^3 u (  \partial_x^4 u \partial_x^2 u +   \partial_x^3 u \partial_x^3 u) dx ,
			\\
			& = - I_2 - \int ( \partial_x^3 u)^3 dx,
		\end{align*}
which lead to		
		\begin{align*}
			|I_2| & \leq  \int | \partial_x^3 u|^3 dx,
			\\
			& \leq \norm{\partial_x^3 u}_{L^\infty } \norm{\partial_x^3 u}_{L^2 }^2,
			\\
			& \lesssim \norm{u}_{ \dot H^4 }^{7/8} \norm{ u}_{L^2 }^{1/8} \norm{u}_{ \dot H^4 }^{2 \cdot 3/4} \norm{ u}_{L^2 }^{2/4},
			\\
			& \lesssim (1+t)^{-2.25 \frac{5}{8}} \tnorm{u}^3.
		\end{align*}
Collecting all the estimates
		\begin{align*}
			\frac{d}{dt} \norm{u(t)}_{\dot H^4}^2 & \lesssim (1+t)^{-2.25 \frac{5}{8}} \tnorm{u}^3 ,
			\\
			\norm{u(t)}_{\dot H^4}^2 & \lesssim  \norm{u_0}_{\dot H^4}^2 + \tnorm{u}^3 \int_0^t  (1+s)^{-2.25 \frac{5}{8}} ,
			\\
			& \lesssim  \norm{u_0}_{\dot H^4}^2 + \tnorm{u}^3 \int_0^\infty  (1+s)^{-2.25 \frac{5}{8}} ,
			\\
			& \lesssim  \norm{u_0}_{\dot H^4}^2 + \tnorm{u}^3.
		\end{align*}
		Thus, 
		\begin{align}
			\norm{u(t)}_{\dot H^4} & \lesssim  \norm{u_0}_{\dot H^4} + \tnorm{u}^{\frac{3}{2}} \label{h4bound}.
		\end{align}

As a conclusion, by \eqref{l2bound} and \eqref{h4bound}, we find
	\begin{equation}
		\tnorm{u} \lesssim  2\norm{u_0}_{\dot H^{4}}+ \tnorm{u}^2 +  \tnorm{u}^{\frac{3}{2}}.
	\end{equation}
	If the initial data $ \norm{u_0}_{\dot H^{4}}$ is small enough, we get that $	\tnorm{u}$ is bounded independently of the choice of $T$. As a consequence, global existence for small initial data is proved.
\end{proof}

\begin{theo}[Existence of solutions for small data]
Let $u_0 \in \dot H^{3}(\TT)$ be a zero mean initial data for the Cauchy problem \eqref{cubiclocal}Then, if $\norm{u_0}_{  \dot H^{3}(\TT)}$ is small enough, there exist a solution 
	$$
	u\in C([0,T],H^3), \forall  \, T >0.
	$$
	Furthermore, the solution verifies
	$$
	\norm{u(t)}_{L^2}\leq C(1+t)^{-1.25}.
	$$
\end{theo}

\begin{proof} 
The proof is based on a polynomial estimate on the energy
	\begin{equation}
		\tnorm{u} = \sup_{t \in [0,T]} \left((1+t)^{1.25} \norm{u}_{L^2}+ \norm{u}_{{ \dot H}^3}\right). 	
	\end{equation}	
	
		Using the Duhamel formula, we can write the solution as
		
		\begin{equation}
			u(x,t) = e^{- \Lambda^{-1} t} u_0(x) + \int_{0}^{t} e^{- \Lambda^{-1} (t-s)} u^2(x,s) u_x(x,s) ds.
		\end{equation}
		Taking the $L^2$ norm, we get that
		\begin{align*}
			\norm{u(t)}_{L^2} & \leq  \norm{e^{- \Lambda^{-1} t} u_0}_{L^2} + \int_{0}^{t} \norm{e^{- \Lambda^{-1} (t-s)} u^2(s) u_x(s)}_{L^2} ds ,
			\\
			& \leq C \norm{u_0}_{ \dot H^{ 1.5}}   (1+t)^{-1.25} + C \int_{0}^{t} (1+t-s)^{-1.25} \norm{u^2(s) u_x(s)}_{ \dot H^
				{1.5}} ds ,
			\\
			&  \lesssim \norm{u_0}_{ \dot H^{  1.5}} (1+t)^{-1.25} \\
			& \quad + \int_{0}^{t} (1+t-s)^{-1.25} \left( \norm{u(s)}_{L^\infty}^2 \norm{ u(s)}_{ \dot H^{  2.5}} + \norm{u^2(s)}_{\dot H^{1.5}} \norm{u_x(s)}_{L^\infty} \right) ds .
\end{align*}
Using interpolation, we find that
		\begin{align*}
			\norm{u(t)}_{L^2} &  \lesssim \norm{u_0}_{ \dot H^{  1.5}} (1+t)^{-1.25} \\
			& \quad  + \int_{0}^{t} (1+t-s)^{-1.25} \left( \norm{u(s)}_{L^2}^{1/6 +5/3}  \norm{u(s)}_{\dot H^{3}}^{5/6 +1/3} + \norm{u(s)}_{L^2}^{1/2 +5/6}  \norm{u(s)}_{\dot H^{3}}^{1/2+1/6+1}   \right) ds ,
			\\
			&  \lesssim \norm{u_0}_{ \dot H^{  1.5}} (1+t)^{-1.25} \\
			& \quad  + \int_{0}^{t} (1+t-s)^{-1.25}  \left( (1+s)^{-1.25 \cdot 11/6}+  (1+s)^{-1.25 \cdot 4/3} \right) \tnorm{u}^3 ds ,
			\\
			&  \lesssim \norm{u_0}_{ \dot H^{ 1.5}} (1+t)^{-1.25} + \tnorm{u}^3 (1+t)^{-1.25} ,
			\\
		\end{align*}
		Then, 
		\begin{equation} \label{l2bound2b}
			(1+t)^{1.25}	\norm{u(t)}_{L^2}  \lesssim  \norm{u_0}_{ \dot H^{1.5}}+ \tnorm{u}^{3}  .
		\end{equation}
		
Now we estimate		
		\begin{align*}
			\frac{1}{2} \frac{d}{dt} \int |\partial_x^3 u|^2 dx & = \int  \partial_x^3 u \partial_x^3 \left( -  u^2  u_x - \Lambda^{-1} u \right) dx ,
			\\
			& = - \int  |\Lambda^{1/2} \partial_x^2 u|^2 -  \int  \partial_x^3 u \partial_x^3 (u^2 u_x) dx ,
			\\
			& \leq -   \int  \partial_x^3 u \partial_x^3 (u^2 u_x) dx ,
			\\
			& = C_1  \int  (\partial_x^3 u)^2 u_x u dx  + C_2  \int  (\partial_x^2 u)^2 \partial_x^3 u u dx  + C_3 \int ( u_x)^2 \partial_x^3 u \partial_x^2u dx ,
			\\
			& =:  C_1 I_1 + C_2  I_2 + C_3 I_3.
		\end{align*}
We now compute that		
		\begin{align*}
			|I_1| & \lesssim \norm{u}_{L^\infty }\norm{u_x}_{L^\infty } \norm{u}_{\dot H^3 }^2, \\
			& \lesssim \norm{u}_{L^2 }^{5/6+1/2} \norm{u}_{\dot H^3 }^{2+1/6+1/2} ,\\
			& \lesssim (1+t)^{-1.25 \cdot 4/3} \tnorm{u}^4.
		\end{align*}
The second integral can be estimated as		
		\begin{align*}
			|I_2| & \leq  \norm{u}_{L^\infty }\norm{\partial_x^3 u}_{L^2 } \norm{(\partial_x^2 u)^2}_{L^2 },\\
			& \lesssim \norm{u}_{L^\infty }\norm{u}_{\dot H^3 } \norm{\partial_x^2 u}_{L^4 }^2,\\
			& \lesssim \norm{u}_{L^2 }^{5/6+1/6}\norm{u}_{\dot H^3 }^{1+1/6+11/6}, \\
			& \lesssim (1+t)^{-1.25} \tnorm{u}^4.
		\end{align*}
Finally, the third integral is bounded as follows		
		\begin{align*}
			|I_3| & \leq  \norm{u_x}_{L^\infty }^2\norm{\partial_x^3 u}_{L^2 } \norm{\partial_x^2 u}_{L^2 },\\
			& \lesssim \norm{u}_{L^2 }^{1+1/3}\norm{u}_{\dot H^3 }^{2+2/3} ,\\
			& \lesssim (1+t)^{-1.25 \cdot 4/3} \tnorm{u}^4.
		\end{align*}
Putting every estimate together we find that		
		\begin{align*}
			\frac{d}{dt} \norm{u(t)}_{\dot H^3}^2 & \lesssim (1+t)^{-1.25} \tnorm{u}^4,
			\\
			& \lesssim  \norm{u_0}_{\dot H^3}^2 + \tnorm{u}^4 \int_0^\infty  (1+s)^{-1.25}.
		\end{align*}
		Thus, 
		\begin{align}
			\norm{u(t)}_{\dot H^3} & \lesssim  \norm{u_0}_{\dot H^3} + \tnorm{u}^{2}. \label{h3bound}
		\end{align}
		
Finally, by \eqref{l2bound2b} and \eqref{h3bound}, 
	\begin{equation}
		\tnorm{u} \lesssim  \norm{u_0}_{\dot H^{3}}+ \tnorm{u}^2 +  \tnorm{u}^{3}.
	\end{equation}
	If the initial data $ \norm{u_0}_{\dot H^{3}}$ is small enough, we get that $	\tnorm{u}$ is bounded independently of the choice of $T$. As a consequence, global existence for small initial data is proved.
\end{proof}

\begin{theo}[Existence of solutions for small data]
Let $u_0 \in H^{3}(\TT)$ be the initial data for the Cauchy problem \eqref{cubicnonlocal2}. Then, if $\norm{u_0}_{H^{3}(\TT)}$ is small enough, there exist a solution 
	$$
	u\in C([0,T],H^3),\; \forall \, T >0.
	$$
	Furthermore, the solution verifies
	$$
	\norm{u(t)}_{L^2}\leq C(1+t)^{-1.25}.
	$$
\end{theo}

\begin{proof} 
We now estimate the previous energy
		\begin{equation}
		\tnorm{u} = \sup_{t \in [0,T]} \left((1+t)^{1.25} \norm{u}_{L^2}+ \norm{u}_{{ \dot H}^3}\right). 	
	\end{equation}
We start with the low regularity norm. Using the Duhamel formula, we can write the solution as
		
		\begin{equation}
			u(x,t) = e^{- \Lambda^{-1} t} u_0(x) - \frac{1}{2}  \int_{0}^{t} e^{- \Lambda^{-1} (t-s)} \mathcal{H} u^2(x,s) u_x(x,s) ds.
		\end{equation}
		Taking the $L^2$ norm, we get
		\begin{align*}
			\norm{u(t)}_{L^2} & \leq  \norm{e^{- \Lambda^{-1} t} u_0}_{L^2} + \frac{1}{2} \int_{0}^{t} \norm{e^{- \Lambda^{-1} (t-s)} \mathcal{H}u^2(s) u_x(s)}_{L^2} ds ,
			\\
			& \leq C \norm{u_0}_{ \dot H^{ 1.5}}   (1+t)^{-1.25} + C \int_{0}^{t} (1+t-s)^{-1.25} \norm{ \mathcal{H}u^2(s) u_x(s)}_{ \dot H^
				{1.5}} ds ,
			\\
			&  \lesssim \norm{u_0}_{ \dot H^{ 1.5}} (1+t)^{-1.25} + \tnorm{u}^3 (1+t)^{-1.25} ,
			\\
		\end{align*}
		Then, 
		\begin{equation} \label{l2bound2c}
			(1+t)^{2.25}	\norm{u(t)}_{L^2}  \lesssim  \norm{u_0}_{ \dot H^{1.5}}+ \tnorm{u}^{3}  .
		\end{equation}
		
Now we continue with the high regularity term:		
		\begin{align*}
			\frac{1}{2} \frac{d}{dt} \int |\partial_x^3 u|^2 dx & = \int  \partial_x^3 u \partial_x^3 \left( -   \mathcal{H}u^2  u_x - \Lambda^{-1} u \right) dx ,
			\\
			& = - \int  |\Lambda^{1/2} \partial_x^2 u|^2 -  \int  \partial_x^3 u \partial_x^3 ( \mathcal{H}u^2 u_x) dx ,
			\\
			& \leq C_1  \int  \partial_x^3 u\partial_x^2 \Lambda u^2 u_x dx  + C_2  \int \partial_x (\partial_x^2 u)^2   \partial_x \Lambda u^2 dx + C_3 \int  \partial_x^3 u \Lambda u^2 \partial_x^3 u dx ,
			\\
			& =:  C_1 I_1 + C_2  I_2 + C_3 I_3 .
		\end{align*}
As before we can estimate		
		\begin{align*}
			|I_1| & \lesssim  \norm{u_x}_{L^\infty } \norm{u^2}_{\dot H^{3}} \norm{u}_{\dot H^{3}}, \\
			& \lesssim  \norm{u_x}_{L^\infty } \norm{u}_{L^\infty }\norm{u}_{\dot H^{3}}^2 ,\\
			& \lesssim  \norm{u}_{L^2 } \norm{u}_{\dot H^{3}}^3,\\
			& \lesssim  (1+t)^{-1.25} \tnorm{u}^4.
		\end{align*}
The second term can be bounded as		
		\begin{align*}
			|I_2| & \lesssim  \norm{ (\partial_x^2 u)^2  }_{\dot H^1} \norm{   u^2  }_{\dot H^2} , \\
			& \lesssim  \norm{ \partial_x^2 u  }_{L^\infty}  \norm{ u  }_{\dot H^3}   \norm{  u  }_{L^\infty}  \norm{ u  }_{\dot H^2} , \\
			& \lesssim \norm{u}_{L^2}^{1+1/3} \norm{u}_{\dot H^3}^{2+2/3},\\
			& \lesssim (1+t)^{-1.25 \cdot 4/3} \tnorm{u}^4.
		\end{align*}
Finally, the third term can be estimated as		
		\begin{align*}
			|I_3| & \lesssim   \norm{ u  }_{ \dot H^3}^2 \norm{  \Lambda u^2  }_{L^\infty}, \\
			& \lesssim   \norm{ u  }_{ \dot H^3}^2 \norm{  u^2  }_{W^{1,\infty}} ,\\
			& \lesssim \norm{ u  }_{ \dot H^3}^2 \norm{u^2}_{L^2}^{1/2} \norm{u^2}_{\dot H^3}^{1/2}, \\
			& \lesssim \norm{ u  }_{ \dot H^3}^{2+1/2} \norm{u}_{L^4} \norm{u}_{L^\infty}^{1/2}, \\
			& \lesssim \norm{ u  }_{ \dot H^3}^{2+1/2+1/12+1/12} \norm{u}_{L^2}^{11/12+5/12} ,\\
			& \lesssim (1+t)^{-1.25 \cdot 4/3} \tnorm{u}^4.
		\end{align*}
Collecting every estimate we conclude that			
		\begin{align*}
			\frac{d}{dt} \norm{u(t)}_{\dot H^3}^2 & \lesssim (1+t)^{-1.25} \tnorm{u}^4,
			\\
			& \lesssim  \norm{u_0}_{\dot H^3}^2 + \tnorm{u}^4 \int_0^\infty  (1+s)^{-1.25}.
		\end{align*}
		Then,
		\begin{align}
			\norm{u(t)}_{\dot H^3} & \lesssim  \norm{u_0}_{\dot H^3} + \tnorm{u}^{2}. 
		\end{align}

	Analogously as in the previous cubic model, these estimates prove global existence for small initial data now in the inhomogeneous $H^3$ due to the lack of hypothesis on the mean of the initial data.	
\end{proof}

\subsection{Finite-time singularities}

\begin{theo}[Blow-up for the quadratic 1D model]
	Let $u(x,t)$ be a smooth solution of the Cauchy problem \eqref{quadratic}. Then, there exists $C$ such that if the initial data $u_0(x)$ satisfies 
	$$ \min_{x} u_x(x,0) \leq -C,$$
	$u_x(x,t)$ blows up in finite time.
\end{theo}

\begin{proof}

We consider the evolution of 
$$
m(t)=\min_{x} u_x(x,t)=u_x(x_t,t).
$$
Using a pointwise argument, we have that
$$
\frac{d}{dt}m(t)=-m^2(t)-\Lambda^{-1}u_x(x_t,t).
$$
We compute
\begin{align*}
	\Lambda^{-1}u_x(x_t)&=-\frac{1}{2\pi}\int_{-\pi}^\pi \log\left(4\sin^2\left(\frac{y}{2}\right) \right)u_x(x_t-y)dy ,\\
	&=-\frac{1}{2\pi}\int_{-\pi}^\pi \log\left(\sin^2\left(\frac{y}{2}\right) \right)u_x(x_t-y)dy ,\\ 
	&=-\frac{1}{2\pi}\int_{-\pi}^\pi \log\left(\sin^2\left(\frac{y}{2}\right) \right)(u_x(x_t-y)-u_x(x_t))dy-\frac{1}{2\pi}\int_{-\pi}^\pi \log\left(\sin^2\left(\frac{y}{2}\right) \right)u_x(x_t)dy ,\\
	&=\frac{1}{2\pi}\int_{-\pi}^\pi \log\left(\sin^2\left(\frac{y}{2}\right) \right)(u_x(x_t)-u_x(x_t-y))dy+\log(4)m(t) ,\\
	&\geq \log(4)m(t).
\end{align*}
From the previous computation we find that
$$
\frac{d}{dt}m(t)\leq -(m(t))^2-\log(4)m(t).
$$
This ODE blows up if $m(0)$ is large enough compared with $\log(4)$, which concludes the result.

\end{proof}

\begin{theo}[Blow-up for the cubic local 1D model]
	Let $u(x,t)$ be a solution of the Cauchy problem \eqref{cubiclocal}. Then, if 
	$$ \int_{-1}^1 \left(- u_0^2(x)+u_0^2(0)\right) \left( |x|^{-\delta} -1\right)  \operatorname{sign}(x) dx\geq \|u_0\|_{L^2}^2,$$
	$\partial_x u(x,t)$ blows-up in finite time.
\end{theo}

\begin{proof} This proof is based on arguments shown in \cite{DDL2009}.
	Firstly, note that we have uniform control of the $L^2$ norm by 
	$$ \norm{u}_{L^2}^2 + 2 \int_0^t \norm{\Lambda^{-1/2} u}_{L^2}^2 ds \leq \norm{u_0}_{L^2}^2.$$
	The $L^\infty$ norm is also controlled via control of the maximum and minimum values of $u$. We define 
	$$ M(t) = \max_{x\in \TT} u(x,t), \quad m(t) = \min_{x\in\TT} u(x,t).$$
	If we assume that the supremum at time $t>0$  is attained at the maximum point, i.e., 
	$$ \norm{u(\cdot,t)}_{L^\infty} = M(t),$$
	then $M(t) \ge 0$ and it holds
	$$ \frac{d}{dt}M(t) \leq \norm{\Lambda^{-1} u}_{L^\infty} \leq C\|u_0\|_{L^2}.$$
Similarly, if the supremum is attained at the minimum point, then $m(t) \leq 0$ and a similar bound holds. Consequently, 
\begin{equation}\label{eq:expgrowth}
		 \norm{u(\cdot,t)}_{L^\infty} \leq \norm{u_0}_{L^\infty} +C\|u_0\|_{L^2}t.
\end{equation}
	
	Let us define 
	$$ W(x,t) = u^2(x+y(t),t), \quad y'(t) =u^2(y(t),t),\quad y(0)=0.$$
	Then, by multiplication of equation (\ref{cubiclocal}a) by $u$, one gets 
	\begin{align*}
		 \frac{d}{dt} W(x,t) & = \partial_t u^2(x+y(t),t) + \partial_x u^2(x+y(t),t) u^2(y(t),t), \\
		 &=\partial_x u^2(x+y(t),t) \left( u^2(y(t),t)-u^2(x+y(t),t) \right) - 2 u(x+y(t),t) \Lambda^{-1} u(x+y(t),t), \\
		 & = \partial_x W(x,t) \left( W(0,t)- W(x,t)\right) -  2 u(x+y(t),t)  \Lambda^{-1} u(x+y(t),t). \\
	\end{align*}
	
	Now, given $\delta < 1/2$, we define the functional
	$$ L(t) = \int_{-1}^1 \left(- W(x,t)+W(0,t)\right) \left( |x|^{-\delta} -1\right)  \operatorname{sign}(x) dx.$$
	Then, the evolution in time of $L(t)$ is given by
	\begin{align*}
		\frac{d}{dt}L(t) &= \int_{-1}^1 \left(- \partial_ tW(x,t)+ \partial_t W(0,t)\right) \left( |x|^{-\delta} -1\right)  \operatorname{sign}(x) dx, \\
		& =  \int_{-1}^1   \partial_x W(x,t) \left( W(x,t)-W(0,t)\right)  \left( |x|^{-\delta} -1\right)  \operatorname{sign}(x) dx, \\
		&\quad + 2\int_{-1}^1 \left(  u(x+y(t),t) \Lambda^{-1} u(x+y(t),t)  - u(y(t),t)  \Lambda^{-1} u(y(t),t)    \right) \left( |x|^{-\delta} -1\right)  \operatorname{sign}(x) dx, \\
		& = I_1 + I_2. 
	\end{align*}
The first term can be estimated as
\begin{align*}
	I_1	& =  \frac{1}{2} \int_{-1}^1    \partial_x \left( W(x,t)-W(0,t)\right)^2 \left( |x|^{-\delta} -1\right)  \operatorname{sign}(x) dx, \\
	& = \frac{1}{2}\left\{  \left( W(x,t)-W(0,t) \right)^2 \left( |x|^{-\delta} -1\right)  \operatorname{sign}(x)  \right\}_{-1}^1, \\
	& \quad - \frac{1}{2}  \int_{-1}^1  \left( W(x,t)- W(0,t)\right)^2  \partial_x \left[   \left( |x|^{-\delta} -1\right)  \operatorname{sign}(x)  \right] dx.
\end{align*}
In fact, we compute
\begin{align*}
	I_1		& =  \delta \frac{1}{2}  \int_{-1}^1  \left( W(x,t)- W(0,t)\right)^2    |x|^{-\delta-1} dx, \\
	& \ge \delta \frac{1}{2} \left( \int_{-1}^1 \left\vert-W(x,t)+ W(0,t)\right\vert |x|^{-\frac{\delta+1}{2}} dx \right)^2, \\
	& \ge  \delta \frac{1}{2} \left( \int_{-1}^1 \left\vert-W(x,t)+ W(0,t)\right\vert  \left( |x|^{-\delta}-1\right)dx \right)^2, \\
	& \ge C(\delta) L(t)^2.
\end{align*}
The second term is bounded by
\begin{align*}
	I_2 &=  2\int_{-1}^1 \left( - u(x+y(t),t) \Lambda^{-1} u(x+y(t),t)  + u(y(t),t)  \Lambda^{-1} u(y(t),t)    \right) \left( |x|^{-\delta} -1\right)  \operatorname{sign}(x) dx, \\
	& \leq \norm{u}_{L^2} \norm{\Lambda^{-1} u}_{L^\infty} \norm{|x|^{-\delta}-1}_{L^2{(-1,1)}}, \\
	& \leq C(\delta) \norm{u_0}_{L^2}^2, \\
\end{align*}
as long as $\delta < 1/2$.

Both estimates provide 
\begin{equation} \label{eq:riccati}
	 \frac{d}{dt} L(t) \ge C(\delta) \left(  L(t)^2- \norm{u_0}_{L^2}^2 \right).
\end{equation}
Assuming that 
$$  L(0)- \norm{u_0}_{L^2}  >0,$$ the inequality \eqref{eq:riccati} implies finite-time blow up of $L(t)$. Note that 
 $$L(t) \lesssim \norm{ \partial_x W(\cdot,t)}_{L^\infty} \lesssim \norm{\partial_x u^2(\cdot,t)}_{L^\infty} \lesssim  \norm{u_x(\cdot,t)}_{L^\infty},$$
 which translates into finite-time blow-up of $u_x(x,t)$.

\end{proof}

\begin{theo}[Blow-up for the cubic non-local 1D model]
	Let $u(x,t)$ be a solution of the Cauchy problem \eqref{cubicnonlocal2}. Then, if the initial data $u_0(x)$ has odd symmetry and we fix $\delta <1/2,$ $\Lambda^{\delta} u(x,t)$ blows up in finite time.
\end{theo}

\begin{proof}
If we repeat the previous pointwise estimate, we find that
	$$ \frac{d}{dt}M(t) \leq \norm{\Lambda^{-1} u}_{L^\infty} \leq CM(t).$$
	Using the same idea for $m(t)$, we conclude
	$$ \norm{u(\cdot,t)}_{L^\infty} \leq \norm{u_0}_{L^\infty} e^{Ct}.$$
As the initial data is odd and such symmetry is preserved by the evolution, the resulting solution is also odd. As a consequence, it conserves the mean as in the original problem \eqref{Stokes}.		
	Let us define 
	$$ J(t) = \int_0^\infty \frac{u^2}{|x|^{1+\delta}} dx .$$
	for some $\delta<\frac{1}{2}$ to be fixed later. 
	Then, by Cordoba-Cordoba-Fontelos inequality \cite{CCF2005}
	\begin{align*}
		\frac{d}{dt} J(t) &=  \int_0^\infty \frac{\partial_t u^2}{|x|^{1+\delta}} dx, \\
		& = - \frac{1}{2} \int_0^\infty \frac{\mathcal{H}(u^2) \partial_x(u^2)}{|x|^{1+\delta}} dx-  2 \int_0^\infty \frac{u \Lambda^{-1}u }{|x|^{1+\delta}} dx,\\
		& \ge \frac{1}{2} \int_0^\infty \frac{u^4}{|x|^{2+\delta}} dx -2  \int_0^\infty \frac{u \left( \Lambda^{-1}u(x)-\Lambda^{-1}u(0) \right) }{|x|^{1+\delta}} dx.
	\end{align*}
	Note that 
	\begin{align*}
		\left\vert  \Lambda^{-1}u(x)-\Lambda^{-1}u(x+h) \right\vert & \leq |h|^{1/2} \seminorm{\Lambda^{-1} u}_{\dot C^{1/2}}, \\
		& \lesssim |h|^{1/2} \norm{u_0}_{L^\infty} e^{Ct} .
	\end{align*}
	Furthermore, 
	\begin{align*}
		-	 \int_0^\infty \frac{u \left( \Lambda^{-1}u(x)-\Lambda^{-1}u(0) \right) }{|x|^{1+\delta}} dx & = - \int_0^L \frac{u \left( \Lambda^{-1}u(x)-\Lambda^{-1}u(0) \right) }{|x|^{1+\delta}} dx-  \int_L^\infty \frac{u \left( \Lambda^{-1}u(x)-\Lambda^{-1}u(0) \right) }{|x|^{1+\delta}} dx, \\
		& \ge -C_1 \norm{u_0}_{L^\infty}^2 e^{2C_2t} \int_0^L \frac{1}{|x|^{1/2 + \delta}} dx - C_1 \norm{u_0}_{L^\infty}^2 e^{2C_2t} \int_T^\infty \frac{1}{x^{1+\delta}} dx, \\
		& \ge - C \norm{u_0}_{L^\infty}^2 e^{Ct}.
	\end{align*}
By Holder inequality, 
$$ \int_0^\infty \frac{u^4}{x^{2+\delta}} dx \ge C \left(\int_0^\infty \frac{u^2}{x^{1+\delta}} dx \right)^2  .$$
	
All together, we have that 
	\begin{equation*}
		\frac{d}{dt} J(t) \ge C(\delta,\norm{u_0}_{L^\infty}) \left( J^2(t) -  e^{Ct} \right).
	\end{equation*}
Then, $J(t)$ blows-up in finite time for certain large enough initial data $\norm{u_0}_{L^\infty}$.

\end{proof}

\subsection*{Acknowledgments.} 		
FG and ES were partially supported by the AEI grants EUR2020-112271, PID2020-114703GB-I00 and PID2022-140494NA-I00. FG was partially supported by the AEI grant RED2022-134784-T funded by MCIN/AEI/10.13039/501100011033, by the Fundacion de Investigacion de la Universidad de Sevilla through the grant FIUS23/0207 and acknowledges support from IMAG, funded by MICINN through the Maria de Maeztu
Excellence Grant CEX2020-001105-M/AEI/10.13039/501100011033. 
RGB is funded by the project "An\'alisis Matem\'atico Aplicado y Ecuaciones Diferenciales" Grant PID2022-141187NB-I00 funded by MCIN /AEI /10.13039/501100011033 / FEDER, UE and acronym "AMAED". This publication is part of the project PID2022-141187NB-I00 funded by MCIN/ AEI /10.13039/501100011033. ES acknowledges support from the Max Planck Institute for Mathematics in the Sciences.

\bibliographystyle{abbrv}

\end{document}